\newfont{\Bbb}{msbm10 scaled\magstephalf}
\documentclass[11 pt,reqno]{amsart}
\usepackage{amsfonts}
\usepackage{amsmath, amsthm, amscd, amsfonts}
\usepackage{amssymb}
\usepackage{amsmath}
\usepackage{xcolor}
\usepackage{amscd}
\usepackage{graphicx}
 \usepackage{epstopdf}
\allowdisplaybreaks[4]
 \newtheorem{thm}{Theorem}[section]
 \newtheorem{cor}[thm]{Corollary}
 \newtheorem{lem}[thm]{Lemma}
 \newtheorem{prop}[thm]{Proposition}
 \theoremstyle{definition}
 \newtheorem{defn}[thm]{Definition}
\theoremstyle{remark}
 \newtheorem{rem}[thm]{Remark}
 \newtheorem{exm}[thm]{Example}
 \numberwithin{equation}{section}
\newcommand{\pf}{\begin{proof}}
\newcommand{\zb}{\end{proof}}

\newcommand{\ma}{\mathcal}

\def\dim{\mathop{\rm dim}\nolimits}

\def\Ker{\mathop{\rm ker}\nolimits}

\begin{document}

\title[Composition operators between Toeplitz kernels]{Composition operators between Toeplitz kernels}
\author[Y. Liang]{Yuxia Liang}
\address{Yuxia Liang \newline School of Mathematical Sciences,
Tianjin Normal University, Tianjin 300387, P.R. China.} \email{liangyx1986@126.com}
\author[J. R. Partington]{Jonathan R. Partington}
\address{Jonathan R. Partington \newline School of Mathematics,
  University of Leeds, Leeds LS2 9JT, United Kingdom.}
 \email{J.R.Partington@leeds.ac.uk}
\begin{abstract}
Recently, it was shown
that the image of a Toeplitz kernel of dimension greater than $1$ under composition by an inner function is nearly $S^*$-invariant if and only if the inner function is an automorphism. Building on this, we determine the minimal Toeplitz kernel containing the image of a Toeplitz kernel under a composition operator  with a general inner symbol, and extend this to weighted composition operators. Specifically, the  corresponding cases for minimal model spaces are also given, thereby extending known work on the action of composition operators on model spaces. Finally, we use the equivalences between Toeplitz kernels  to derive the explicit maximal vectors for several Toeplitz kernels, with symbols expressed in terms of composition operators and inner functions.
\end{abstract}

\maketitle

\section{Preliminaries}
Let $H^2:=H^2(\mathbb{D})$ denote the standard Hardy space of the unit disc $\mathbb{D},$ which embeds isometrically into $L^2(\mathbb{T})$,  where $\mathbb{T}$ is the unit circle endowed with the normalized Lebesgue measure $m$.
It holds that
 $L^2(\mathbb{T})=H^2 \oplus \overline{H_0^2 },
$ with $H^2 =\bigvee\{z^n:\;n\geq 0\}$ and $\overline{H_0^2}= \bigvee\{z^n:\;n< 0\}= \overline{z H^2}$. Let $ H^\infty:=H^\infty(\mathbb{D})$ be the Banach algebra of bounded analytic functions on $\mathbb{D},$ of which the set of invertible elements is denoted by $\mathcal{G}H^\infty,$ similarly for $L^\infty:=L^\infty(\mathbb{T})$ and $\mathcal{G}L^\infty$. Additionally,  the Smirnov class (see, e.g. \cite[Lemma 4.4.2]{Nik}) is defined as  $$\mathcal{N}_+=\{f\in H(\mathbb{D}):\; \mbox{there exist}\; f_1,\;f_2\in H^\infty \;\mbox{such that}\;f=f_1/f_2\;\mbox{and}\;f_2\;\mbox{is outer}\}.$$
Here an outer function $f\in H^2$ is one that can be written in the form
$$f(z)=\alpha\exp\left(\frac{1}{2\pi}\int_0^{2\pi} \frac{e^{iw}+z}{e^{iw}-z}k(e^{iw})dw\right), \quad z\in \mathbb{D},$$
where $k$ is a real-valued integrable function and $|\alpha|=1.$

Given $g\in L^\infty,$  the Toeplitz operator $T_g: \;H^2\rightarrow H^2$ is   $$(T_g f)(\lambda)=P_{H^2}(g\cdot f)(\lambda)=\int_{\mathbb{T}}\frac{g(\zeta)f(\zeta)}{1-\overline{\zeta}
\lambda}dm(\zeta),$$ where $P_{H^2}$ is the orthogonal projection from $L^2(\mathbb{T})$ onto $H^2$.  In particular, the Toeplitz operator $T_z$ is the unilateral shift  $S$ on $H^2$   defined by $[Sf](z)=zf(z).$ Its adjoint operator on $H^2$ is the backward shift  $[S^*f](z)= (f(z)-f(0))/z.$
The well-known Beurling theorem characterises the nontrivial subspaces of  $H^2$ that are invariant under $S$ as being of the form $\theta H^2,$ where $\theta$ is an inner function, i.e.,  $\theta\in H^\infty$ such that its  modulus is $1$ almost everywhere on $\mathbb{T}.$ Meanwhile, the model spaces $K_\theta:=H^2\ominus \theta H^2=\ker T_{\overline\theta}$ are $S^*$-invariant.

With the development of  $S$- or $S^*$-invariant subspaces in spaces on various domains, the concept of near invariance emerges.
Recall that a closed subspace $\ma{M}\subseteq H^2$ is said to be nearly $S^*$-invariant if whenever $f\in \ma{M}$ and $f(0)=0,$ then $S^*f\in\ma{M}.$ That is, $\ma{M}\subseteq H^2$ is  nearly $S^*$-invariant if the zeros of functions in $\ma{M}$ can be divided out without leaving the space.
The study of nearly $S^*$-invariant subspaces in $H^2$ was first explored by Hayashi  in \cite{Ha}, Hitt in \cite{hitt}, and then Sarason \cite{Sa1,Sa2} in relation to the kernels of Toeplitz operators. Subsequently, C\^{a}mara and Partington conducted  a systematic investigation  of near invariance and Toeplitz kernels (see, e.g. \cite{CaP1, CaP2}). After that we extended the near invariance  to left invertible operators on Hilbert space and the shift semigroup on $L^2(0,\infty)$ (see, e.g. \cite{LP2,LP3}).

Hitt proved the following well-known characterization of nearly $S^*$-invariant subspaces in $H^2$, with a vectorial generalization formulated in \cite[Theorem 4.4]{CCP10}.

\begin{thm}\cite[Proposition 3]{hitt} \label{thm Hitt}The nearly $S^*$-invariant subspaces of $H^2$ have the form $\ma{M}=uK$, with $u\in \ma{M}$ of unit norm, $u(0)>0,$ and $u$ orthogonal to all elements of $\ma{M}$ vanishing at the origin, $K$  an $S^*$-invariant subspace, and the operator of multiplication by $u$  isometric from $K$ into $H^2$.
\end{thm}

The kernels of Toeplitz operators on $H^2$ are always nearly $S^*$-invariant. Model spaces, a specific type of Toeplitz kernel, are studied in relation to topics like  truncated Toeplitz operators,  Schr\"{o}dinger operators, Hankel operators and classical extremal problems (see, e.g. \cite[Chapter 5]{GMR}).  Recently, Fricain, Hartmann and Ross \cite{FHR} provided a necessary and sufficient condition for a function $F$ to multiply one model space $K_\theta$ into another model space. This is also related with the important work of Crofoot in \cite{Cro}.
 C\^{a}mara and Partington  later explored model spaces and Toeplitz operators in  general $H^p$ spaces on the upper half-plane (\cite{CMP4}), especially posing and characterising the maximal vectors in a model space. In \cite{CaP3}, they introduced a new perspective by considering multipliers between Toeplitz kernels in terms of the maximal vectors and  surjective multipliers (see Theorem \ref{wequalkernel}).

Composition operators on the Hardy space are fundamental objects in operator theory, with connections to complex analysis,  functional analysis and linear dynamics (see, e.g., \cite{CM1995,Ma,MSJ}). For composition operators on model spaces, Mashreghi and Shabankhah \cite{MS}
provided a complete description of bounded composition operators on model spaces $K_B$ for a finite Blaschke product $B$, showing that these  composition operators form a group on $K_B$.  They also characterized the minimal model space $K_v$ containing the image $C_\psi (K_\theta)$ in the following theorem. Here the minimal model space $K_v$ containing $C_\psi (K_\theta)$ means that $v$   divides all other inner functions with the inclusion property.

\begin{thm}\cite[Theorem 2.1]{MS1} \label{containing-model} Let $\psi$ and $\theta$ be   inner functions,  and let
\begin{align*}v(z)=\left\{                           \begin{array}{ll}
                               \theta\circ \psi, & \;\theta(0)\neq 0\;\mbox{and}\; \psi(0)=0,\\ z(\theta\circ \psi), & \;\theta(0)\neq 0\;\mbox{and}\; \psi(0)\neq0, \\
                               z(\theta\circ \psi)/\psi, &\;  \theta(0)=0.
                             \end{array}
                           \right.
 \end{align*}Then the mapping $C_\psi:\;K_\theta\rightarrow K_v$ is well-defined and bounded. 	Moreover, $K_v$ is the minimal model space containing  $C_\psi(K_\theta).$\end{thm}

 We generalize this result in Theorems \ref{thm:3.8} and \ref{thm uCpsi}.

Model spaces are themselves  Toeplitz kernels,  and we  recently proved that if $\psi$ is an automorphism, then the image of any Toeplitz kernel under $C_\psi$ remains a Toeplitz kernel, thus nearly $S^*$-invariant.

\begin{thm} \cite[Theorem 2.1]{LP4}\label{psiT}
Let $\psi$ be an automorphism and $\Ker T_F$ be a Toeplitz kernel with $F\in L^\infty$; then it follows that
$$C_\psi(\Ker T_F )= \Ker T_G, \;\mbox{where} \;G= (F\circ \psi)\frac{\psi}{z}.$$\end{thm}

However, the next theorem shows that if the inner function $\psi$ is not an automorphism, the image of the Toeplitz kernel under  $C_\psi$ no longer retains near $S^*$-invariance.

\begin{thm}\label{thm CphiM}\cite[Theorem 2.11]{LP4}
Let $\mathcal{M}$ be a subspace of $H^2$ of dimension at least $2$, and let $\psi$ be an inner function that is not an automorphism on $\mathbb{D}$. Then $C_\psi(\mathcal{M})$ is not nearly $S^*$-invariant (and thus not a Toeplitz kernel).\end{thm}

These two results imply that for an $F\in L^\infty$ such that $\dim \ker T_F > 1$, and an inner function  $\psi$,  $C_\psi(\Ker T_F)$ is nearly $S^*$-invariant if and only if $\psi$ is an automorphism on $\mathbb{D}$. Especially, Theorem  \ref{thm CphiM} says that otherwise $C_\psi(\Ker T_F)$ is no longer a Toeplitz kernel and therefore motivates us to find the minimal Toeplitz kernel containing it when $\psi$ is an arbitrary inner function.

In this paper, we address the aforementioned question  in three distinct sections. In the main result, Theorem \ref{minimal kernel}, we prove the minimal Toeplitz kernel with symbol $(F\circ \psi)\psi/z$   contains  $C_\psi(\Ker T_F)$ for an inner  $\psi$ and $F\in L^\infty\setminus\{0\}$, while also describing the relations between their maximal vectors.
As a special case, this provides an alternative and concise proof of the key result, Theorem \ref{containing-model}, originally established by Mashreghi and Shabankhah. In Section 3, building on \cite[Theorem 5.1]{CaP5}, we characterize the equivalent condition that ensures   $u\Ker T_F$ is a Toeplitz kernel through Carleson measures in Proposition  \ref{prop equal}. We then extend this by combining the multiplier and composition operators  to formulate the minimal Toeplitz kernels containing $C_\psi(u\Ker T_F)$ or $uC_\psi(\Ker T_F)$. These lead to  two extensions of  Theorem \ref{containing-model} in Theorems \ref{thm:3.8} and \ref{thm uCpsi}. Finally, in Section 4, we use the equivalences between Toeplitz kernels  to derive several results concerning the maximal vectors of Toeplitz kernels. Specially, we provide an example  showing that the sum of two Toeplitz kernels is not necessarily nearly $S^*$-invariant when their intersection is $\{0\}.$

\section{Minimal Toeplitz kernel containing $C_\psi(\Ker T_F)$ for an inner  $\psi$}
In this section, we first review the definitions of the minimal Toeplitz kernel containing a vector and the maximal vector for a Toeplitz kernel (including model spaces). We then consider the same ideas for the image of a Toeplitz kernel under a composition operator.

\begin{defn} \cite[Definition 2.1]{CaP3} For a function $k\in H^2\setminus\{0\}$, we write ${\rm K_{min}}(k)$ for the minimal Toeplitz kernel containing $k,$ that is, ${\rm K_{min}}(k)=\Ker T_g$ for some $g\in L^\infty,$ with $k\in {\rm K_{min}}(k),$ while $\Ker T_g\subseteq \Ker T_h$ for every $h\in L^\infty$ such that $k\in \ker T_h.$ Meanwhile, $k$ is called a maximal vector (or a maximal function) for the Toeplitz kernel $\Ker T_g$. \end{defn}

The existence of minimal kernels and maximal vectors in $H^2$ was presented in the following theorem, of which the upper half-plane case can be found in \cite[Theorem 5.1 and Corollary 5.1]{CaP1}.
\begin{thm} \label{mini kernel}\cite[Theorem 3.3]{CaP2} Let $h\in H^2\setminus \{0\}$ and $h=IO$ be its inner-outer factorization. Then there exists a minimal Toeplitz kernel ${\rm K_{min}}(h)$ containing $h$, namely
$${\rm K_{min}}(h)=\Ker T_{\overline{z} \overline{IO}/O}$$ and $h$ is a maximal vector for  ${\rm K_{min}}(h)$.
\end{thm}

The definition of the minimal Toeplitz kernel containing a subspace can be stated similarly. Additionally, all the maximal vectors for a Toeplitz kernel can be characterized as follows.

 \begin{thm}\cite[Theorem 2.2]{CaP3}\label{maximal1} Let $g\in L^\infty\setminus\{0\}$ be such that $\Ker T_g\neq \{0\}.$ Then $k$ is a maximal vector for $\Ker T_g$ if and only if $k\in H^2$ and $k=g^{-1}\overline{z}\bar{p}$, where $p$ is outer in $H^2.$\end{thm}

\subsection{The minimal Toeplitz kernel containg  $C_\psi (\Ker T_F)$.}
In this subsection, we examine the transformation of Toeplitz kernels under composition operators with inner function symbols and derive explicit formulas for the minimal Toeplitz kernels that contain  their images. Meanwhile, the correspondence of maximal vectors between Toeplitz kernels is characterized.

In our proof, the following theorem  plays a crucial role in establishing the equivalent conditions for the equalities of Toeplitz kernels.

\begin{thm}\cite[Corollary 7.9]{CaP2}\label{equalkernel}  Let $g, h\in L^\infty\setminus\{0\}$  be such that $\Ker T_g$ and $\Ker T_h$ are nontrivial. Then $\Ker T_g=\Ker T_h$ if and only if $$\frac{g}{h}=\frac{\overline{p}}{\overline{q}},\;\;p, q\in H^2 \;\mbox{outer}.$$  If moreover $hg^{-1}\in \mathcal{G}L^\infty,$ then we have $$\Ker T_g=\Ker T_h\;\;\mbox{if and only if}\;\;hg^{-1}\in \mathcal{G}\overline{H^\infty}.$$ \end{thm}

We can now present our first main result.
\begin{thm}\label{minimal kernel} Let $\psi$ be inner  and $F\in L^\infty\setminus\{0\}$ such that  $\Ker T_F\neq \{0\}$. Then there exists a minimal Toeplitz kernel  containing $C_\psi (\Ker T_F)$,
which is $\Ker T_{(F\circ \psi)\psi/z}$. Moreover, if $k$ is a maximal vector  for  $\Ker T_F$,
then $k\circ \psi$ is a maximal vector for  $\Ker T_{(F\circ \psi)\psi/z}$.
Further,  $C_\psi (\Ker T_F)= \Ker T_{(F\circ \psi)\psi/z}$ if and only if $\psi$ is an automorphism on $\mathbb{D}$. \end{thm}

\begin{proof} With $\Ker T_F\neq \{0\}$ and  $k$   a maximal vector for $\Ker T_F$, we decompose $k=IO$ with $I$ inner and $O$ outer in $H^2$.  Then Theorem \ref{mini kernel} implies that
\begin{align}\Ker T_F=\Ker T_{\overline{z}\overline{IO}/O}.\label{kerF}\end{align}
Next we consider the minimal kernel  ${\rm K_{ min}}(k\circ \psi)$ containing $k\circ \psi=(I\circ \psi)(O\circ \psi)\in C_\psi(\Ker T_F)\subseteq H^2.$ \cite[Proposition 6.5]{GMR} implies that $I\circ \psi$ is inner and \cite[Theorem 4.27]{CHL} implies that $O\circ \psi$ is outer, so we have
\begin{align}{\rm K_{ min}}(k\circ \psi)=\ker T_{\frac{\overline{z}\overline{k\circ \psi}}{O\circ \psi}}.\label{mink}\end{align}

 Now take any  nonzero element $h$ of $\Ker T_F,$ then  \eqref{kerF} implies that
$$\frac{\overline{z}\overline{IO}}{O}h\in \overline{zH^2}.$$  We deduce that $$\frac{ (\overline{I\circ \psi})(\overline{ O\circ\psi})}{O\circ \psi} (h\circ \psi)\in \overline{ H^2},$$
 and then we  further  derive that
\begin{align*} \Big(\frac{\overline{z}\overline{k\circ \psi}}{O\circ \psi}\Big)(h\circ \psi) =\Big(\overline{z}
\frac{(\overline{I\circ \psi})(\overline{O\circ \psi})}{O\circ \psi}\Big)(h\circ \psi)\in  \overline{z H^2}. \end{align*}
This indicates that $C_\psi(h)\in {\rm K_{min}}(k\circ \psi)$ for any  $h\in \Ker T_F,$ implying \begin{align*}C_\psi(\Ker T_F)\subseteq {\rm K_{ min}}(k\circ \psi).\label{subset} \end{align*}

Since $C_\psi(k)\in C_\psi(\Ker T_F),$ any Toeplitz kernel containing $C_\psi(\Ker T_F)$ must  contain ${\rm K_{ min}}(k\circ \psi)$. Thus we conclude that ${\rm K_{ min}}(k\circ \psi)$ is the minimal kernel containing $C_\psi(\Ker T_F)$.

Finally, using \eqref{kerF}, Theorem \ref{equalkernel} yields that
$$F=\overline{z}\frac{\overline{IO}}{O}\frac{\overline{p}}{\overline{q}}$$ with $p, q\in  H^2$ outer. This further implies $$zF= \frac{\overline{IO}}{O}\frac{\overline{p}}{\overline{q}}$$ with $p, q\in  H^2$ outer. Therefore,
\begin{eqnarray}\Big(\frac{\overline{z} \overline{k\circ \psi}}{O\circ \psi}\Big) \frac{\overline{p\circ\psi}}{\overline{q \circ \psi}}
= \frac{\overline{z} (\overline{I\circ \psi}) (\overline{O\circ \psi})}{O\circ \psi} \frac{\overline{p\circ\psi}}{\overline{q \circ \psi}}=(F \circ \psi)\frac{\psi}{z}. \label{zFpsi}\end{eqnarray}

\cite[Theorem 4.27]{CHL} shows  $p \circ\psi$ and $ q\circ\psi$ are outer functions. Using  \eqref{mink}, \eqref{zFpsi} and Theorem \ref{equalkernel}, we obtain that $${\rm K_{ min}}(k\circ \psi)=\Ker T_{( F\circ \psi)\psi/z},$$ which is the minimal Toeplitz kernel containing $C_\psi (\Ker T_F).$
Moreover, $k\circ \psi$ is a maximal vector for $\Ker T_{( F\circ \psi)\psi/z}.$

Meanwhile Theorems \ref{psiT} and \ref{thm CphiM} imply
$C_\psi (\Ker T_F)= \Ker T_{(F\circ \psi)\psi/z}$ if and only if $\psi$ is an automorphism on $\mathbb{D}$, ending the proof.\end{proof}

Letting $F=\overline{\theta}$  in Theorem \ref{minimal kernel}, some corollaries on model spaces  follow.

\begin{cor} \label{cor model} Let $\psi$ and $\theta$ be two inner functions such that $K_\theta\neq \{0\}$; then there is a minimal Toeplitz kernel $\Ker T_{(\overline{\theta}\circ \psi) \psi/z}$ containing $C_\psi(K_\theta)$ and  equality holds if and only if $\psi$ is an automorphism.  For an automorphism $\psi$, it holds that
$$ \Ker  T_{(\overline{\theta}\circ \psi)\psi/z}=\frac{1}{\sqrt{\psi'}}K_{\theta\circ \psi}.$$  \end{cor}

Here we should recall a proposition concerning the inclusion.
\begin{prop}\label{prop FG}\cite[Proposition 2.16]{CaP3}
Let $g, h\in L^\infty\setminus\{0\}$ such that $\Ker T_g$ and $\Ker T_h$ are nontrivial. Then the following conditions are equivalent.\\
$({\rm i})$ $\Ker T_g\subseteq \Ker T_h;$\\
$({\rm ii})$ $hg^{-1}\in \overline{\mathcal{N}_+};$\\
$({\rm iii})$ there exists a maximal vector $k$ for $\Ker T_g$  such that $k\in \Ker T_h.$

If moreover $\Ker T_g$ contains a maximal vector $k$ with $k, k^{-1}\in L^\infty,$ then each of the above conditions is equivalent to\\
$({\rm iv})$ \;$k\in \Ker T_h\cap H^\infty.$
\end{prop}
Using Corollary \ref{cor model}  and Proposition \ref{prop FG}, we present a new proof of  Theorem \ref{containing-model}, a result  also  found in \cite[Theorem 6.4]{GMR}.

\pf Since  $K_v$ is the minimal model space  containing $C_\psi(K_\theta)$,  Corollary \ref{cor model} implies that  $\Ker T_{(\overline{\theta}\circ \psi)\psi/z} \subseteq K_v=\Ker T_{\overline{v}}.$ By Proposition \ref{prop FG}, we derive that $$ \Ker T_{(\overline{\theta}\circ \psi)\psi/z} \subseteq \Ker T_{\overline{v}}\;\;\mbox{if and only if}\; \;\frac{\overline{v}z}{( \overline{\theta} \circ \psi)\psi}\in \overline{\mathcal{N}_+}.$$
This means   $$\frac{ v \psi}{z( \theta \circ \psi) }\in \mathcal{N}_+,$$ which is unimodular. By the definition of the Smirnov class, it follows that $$\frac{ v \psi}{z( \theta \circ \psi) }\in H^\infty \subseteq H^2.$$ Thus we seek the minimal inner function $\phi$ such that
$$v=z(\theta \circ \psi)\phi/\psi \in H^2.$$

For the case $\theta(0)\neq 0$ and $\psi(0)=0,$ then $z$ divides $\psi$ and so we take $\phi(z)=\psi(z)/z$, which gives $v(z)=(\theta\circ\psi)(z).$

For the case $\theta(0)\neq 0$ and $\psi(0)\neq 0,$ then we take $\phi=\psi$, which gives $v(z)=z(\theta\circ\psi)(z).$

For the case $\theta(0)= 0$, then $\psi$ cancels  and so we take $\phi=1$, which gives $v(z)=z(\theta\circ\psi)(z)/\psi(z),$ ending the proof.
\zb

The following proposition describes a surjective composition operator between Toeplitz kernels.

\begin{prop}\label{prop sur}  Let $F\in L^\infty\setminus\{0\}$ such that $\Ker T_F\neq \{0\}$ and $\psi$ be inner. Let $k$ denote a maximal vector for $\Ker T_F$ and suppose $C_\psi(\Ker T_F)$ is  also a Toeplitz kernel, then  $C_\psi(\Ker T_F)={\rm K_{min}}(k\circ \psi)$. This holds if and only if $\psi$ is an automorphism on $\mathbb{D}.$\end{prop}

\begin{proof} Let $G\in  L^\infty$ be such that $C_\psi(\Ker T_F)=\Ker T_G.$ We have $k\circ \psi\in \Ker T_G$ and then ${\rm K_{\min}}(k\circ \psi)\subseteq \Ker T_G$. Theorem \ref{minimal kernel} implies that \[
\Ker T_G=C_\psi(\Ker T_F)\subseteq\Ker T_{(F\circ \psi) \psi /z}={\rm K_{\min}}(k\circ \psi).
\]
 Hence $\Ker T_G={\rm K_{\min}}(k\circ \psi).$ That is, $C_\psi(\Ker T_F)={\rm K_{min}}(k\circ \psi)=\Ker T_{(F\circ \psi) \psi /z}$, which is equivalent to $\psi$ being an automorphism on $\mathbb{D}$, ending the proof.
 \end{proof}

\begin{rem}\label{rem S*theta} Since $S^*\theta=(\theta-\theta(0))/z=\theta\overline{zp}$ with an outer and invertible function $p=1-\overline{\theta(0)}\theta,$ Theorem \ref{maximal1} implies
that $S^*\theta$ is a maximal vector for $K_\theta$. \end{rem}

So letting $F=\overline{\theta}$ and $k=S^*\theta$ in Proposition \ref{prop sur}, we deduce the minimal Toeplitz kernel containing  $C_\psi(S^*\theta)$.
\begin{cor}\label{cor models}  Let $\theta$  be an inner function such that $K_\theta\neq \{0\}$ and $\psi$ be an automorphism, then  $$K_{\min}(C_\psi(S^*\theta))=C_\psi(K_\theta)=\Ker T_{(\overline{\theta}\circ \psi)\psi/z}.$$\end{cor}

Motivated by Proposition \ref{prop FG} and Theorem \ref{minimal kernel}, we turn to characterizing the non-surjective composition operators between Toeplitz kernels.
\begin{prop}\label{prop FGPSI} Let $\psi$ be inner and $F, H\in L^\infty\setminus\{0\}$ such that   $\Ker T_F$ and $\Ker T_H$ are nontrivial. Then the following conditions are equivalent.\\
$({\rm i})$ $C_\psi(\Ker T_F)\subseteq \Ker T_H;$\\
$({\rm ii})$ $H(F^{-1}\circ \psi)z/\psi\in \overline{\mathcal{N}_+};$\\
$({\rm iii})$ there exists a maximal vector $k$ for $\Ker T_F$ such that $k\circ \psi\in \Ker T_H.$

If moreover $\Ker T_F$ contains a maximal vector $k$ with $k, k^{-1}\in L^\infty,$ then each of the above conditions is equivalent to\\
$({\rm iv})$ \;$k\circ \psi\in \Ker T_H\cap H^\infty.$
\end{prop}

\begin{proof} Theorem \ref{minimal kernel} implies $C_\psi(\Ker T_F) \subseteq \Ker T_{(F\circ \psi)\psi/z},$ which is the minimal Toeplitz kernel containing $C_\psi(\Ker T_F)$. So $C_\psi(\Ker T_F) \subseteq \Ker T_H$ is equivalent to $\Ker T_{(F\circ \psi)\psi/z}\subseteq \Ker T_H.$ Applying Proposition \ref{prop FG} with $g=(F\circ \psi)\psi/z$ and $h=H$, we obtain the required conclusions.
\end{proof}
We note if $h\in L^\infty\cap \overline{\mathcal{N}_+}$, then  $h\in \overline{H^\infty},$ so the following corollary holds.
\begin{cor} With the same assumptions as in Proposition \ref{prop FGPSI}, if $H(F^{-1}\circ \psi)z/\psi\in L^\infty$, then
$C_\psi(\Ker T_F)\subseteq \Ker T_H$ if and only if $H(F^{-1}\circ \psi)z/\psi\in \overline{H^\infty}.$
\end{cor}
Letting $H=F$ in Proposition \ref{prop FGPSI}, there follows a corollary.
\begin{cor} Let $\psi$ be inner and $F\in L^\infty\setminus \{0\}$ such that  $\Ker T_F$ is nontrivial. If $F (F^{-1}\circ \psi) z/\psi\in \overline{\mathcal{N}_+}$ or $k\circ \psi\in \Ker T_F$ for a maximal vector $k\in \Ker T_F$, then $C_\psi(\Ker T_F)\subseteq \Ker T_F.$\end{cor}

Now choosing $k=S^*\theta$ with $k, \;k^{-1} \in L^\infty,$  Proposition \ref{prop FGPSI} implies the corollary below.

\begin{cor}\label{cor ktheta} Let $\psi, \theta$ be inner and $H\in L^\infty\setminus\{0\}$ such that $K_\theta$ and $\Ker T_H$ are nontrivial. Then the following conditions are equivalent.\\
$({\rm i})$ $C_\psi(K_\theta)\subseteq \Ker T_H;$\\
$({\rm ii})$ $H(\theta\circ \psi)z/\psi\in \overline{\mathcal{N}_+};$\\
$({\rm iii})$    $C_\psi(S^* \theta)\in \Ker T_H\cap H^\infty.$
\end{cor}
\begin{rem}The function $\frac{\theta(z)-\theta(w)}{z-w}$ is also a maximal vector for $K_\theta$ for every $w\in \mathbb{D}.$ So $C_\psi(K_\theta)\subseteq \Ker T_H$ if and only if $\frac{\theta\circ \psi(z)-\theta(w)}{\psi(z)-w}\in \Ker T_H$ for some $w\in \mathbb{D}.$  \end{rem}

\section{Minimal Toeplitz kernel containing $C_\psi(u\Ker T_F)$ and $uC_\psi(\Ker T_F)$}
In this section,  we examine the action of both the multiplier and composition operator on Toeplitz kernels. The explicit formulas for the minimal Toeplitz kernels containing $C_\psi(u\Ker T_F)$ and $uC_\psi(\Ker T_F)$ are  provided in Theorems \ref{thm uTg1} and \ref{thm uTg2}, which indicate two extensions of  Theorem \ref{containing-model} found in Theorems \ref{thm:3.8} and \ref{thm uCpsi}.
These results are based on the following theorem  by C\^{a}mara and Partington on the minimal Toeplitz kernel containing $u\Ker T_g$.

\begin{thm} \cite[Theorem 5.1]{CaP5}\label{thm uTg} Let $g\in L^\infty\setminus\{0\}$ and $u\in H^2$ such that $\Ker T_g\neq \{0\}$ and $u\Ker T_g\subseteq H^2.$ Then there exists a minimal Toeplitz kernel containing $u\Ker T_g$, which is $\Ker T_{g\overline{u}/u_o},$ where $u_o$ is the outer factor of $u.$ Moreover, if $\phi_m$ is a maximal function for $\Ker T_g,$ then $u\phi_m$ is a maximal function for $\Ker T_{g \overline{u}/u_o}.$\end{thm}

It is natural to ask when $u\in H(\mathbb{D})$ acts as a surjective multiplier between two Toeplitz kernels, which is addressed below.

\begin{thm}\cite[Theorem 3.2]{CaP3}\label{wequalkernel}
Let $g, h\in L^\infty\setminus\{0\}$ such that $\Ker T_g$ and $\Ker T_h$ are nontrivial. Then a function $u\in H(\mathbb{D})$ satisfies
$u\Ker T_g=\Ker T_h$ if and only if \\
$(i)$ $u\in \mathcal{C}(\Ker T_g)$ and $u^{-1}\in \mathcal{C}(\Ker T_h)$;\\
$(ii)$ for some (or indeed, for every) maximal vector $k\in \Ker T_g,$ the function $uk$ is a maximal vector for $\Ker T_h.$
 \end{thm}
Here $u\in \mathcal{C}(\Ker T_g)$ means that $u\Ker T_g\subseteq L^2(\mathbb{T}).$  Particularly, $u$ multiplies a model space $K_\theta$ into $H^2$ if and only if $|u|^2 dm$ is a Carleson measure for $K_\theta$. Since $1-\overline{\theta(0)}\theta\in K_\theta \cap H^\infty$ is invertible, it follows that $u\in H^2.$ However, a multiplier between two general Toeplitz kernels may not lie in $H^2$ (see Example \ref{exm k} with $m=2$).

\subsection{The conditions for $u\Ker T_F =\Ker T_{F\overline{u}/u_o}$} In the following proposition, we use Theorem \ref{wequalkernel} to provide an equivalent characterization of when $u\Ker T_F$ is a Toeplitz kernel.
\begin{prop}\label{prop equalFu} Let $F\in L^\infty\setminus\{0\}$ and $u\in H^2$ be outer such that $\Ker T_F\neq \{0\}$ and $u\Ker T_F\subseteq H^2.$ Then  $u\Ker T_F=\Ker T_{F\overline{u}/u}$ if and only if $u\in \mathcal{C}(\Ker T_F)$ and $u^{-1}\in \mathcal{C}(\Ker T_{F\overline{u}/u}).$\end{prop}

It is known that $u\Ker T_F$ is always closed if $u^{-1}\in L^\infty$.  In general, $u\Ker T_F$ is not always closed,  even when $u$ is outer: see the example $(1+z)K_{z\phi^\delta}$ in \cite[Proposition 3.2]{LP3} with  the singular  inner function  $\phi^\delta(z)=\exp(-\delta(1+z)/(1-z))$ with atom at $1$ and $\delta>0$. Since $u\Ker T_F$ cannot be a Toeplitz kernel when $u$ has an inner factor, we formulate a more strict condition on $u$ such that $u\Ker T_F$ is a Toeplitz kernel.

\begin{prop}\label{prop equal} Let $F\in L^\infty\setminus\{0\}$ and $u\in H^2$ be outer such that $\Ker T_F\neq \{0\}$ and $u\Ker T_F\subseteq H^2.$ If  $u^{-1}\in H^\infty,$ it holds that
\begin{eqnarray}u\Ker T_F=\Ker T_{F\overline{u}/u}=\Ker T_{F/u}.\label{equal}\end{eqnarray}\end{prop}

\begin{proof}First, Theorem \ref{thm uTg} implies $u\Ker T_F\subseteq \Ker T_{F\overline{u}/u}=\Ker T_{F/u},$ where the equality follows from the fact $u^{-1}\in H^\infty$  and Theorem \ref{equalkernel}. For the converse inclusion, suppose $f\in\Ker T_{F\overline{u}/u}=\Ker T_{F/u}$, then it holds that
\begin{align*}F\left(\frac{f}{u}\right)=\left(\frac{F}{u}\right)f \in \overline{zH^2},\end{align*} which implies $f/u\in \Ker T_F$ due to the fact  $u^{-1}\in H^\infty$ (which is the   multiplier algebra of $H^2$).  And then $f\in u\Ker T_F.$ In sum, \eqref{equal} holds.
  \end{proof}

As a slight generalization of the example in \cite[page 560]{CaP3}, we provide an outer function $u$ satisfying Proposition \ref{prop equalFu}, but $u^{-1}\notin H^\infty$ such that $u\Ker T_F$ is a Toeplitz kernel.

\begin{exm} \label{exm k}Let $u(z)=(z-1)^{\frac{1}{m}}$ ($m\geq 2$) be the principal branch and $F(z)=z^{-\frac{m+1}{m}}$ with $\arg z\in [0,2\pi)$ on $\mathbb{T}$, then it holds that
$$u\Ker T_F=\Ker T_{z^{-\frac{m+1}{m}}\left(\frac{\bar{z}-1}{z-1}\right)^{\frac{1}{m}}}=K_{z}.$$
  \end{exm}

\subsection{The minimal Toeplitz kernel containing $C_\psi(u\Ker T_F)$} Based on the condition $u\Ker T_F\subseteq H^2,$ we formulate the minimal Toeplitz kernel containing $C_\psi(u\Ker T_F)$ in the theorem below.

\begin{thm} \label{thm uTg1}  Let $F\in L^\infty\setminus\{0\}$ and $u\in H^2$ such that $\Ker T_F\neq \{0\}$ and $u\Ker T_F\subseteq H^2.$ For any inner function $\psi,$  there exists a minimal Toeplitz kernel $\Ker T_{H}$ containing $C_\psi(u\Ker T_F)$, where \begin{eqnarray}H=\psi (F\circ\psi) \frac{ \overline{u}\circ \psi} {z (u_o\circ\psi) }\label{hpsi}\end{eqnarray}and $u_o$ is the outer factor of $u.$ Moreover, if $\phi_m$ is a maximal vector for $\Ker T_F,$ then $C_\psi(u\phi_m)$ is a maximal vector for $\Ker T_H.$\end{thm}

\begin{proof}  Since $\phi_m$ is a maximal vector for $\Ker T_F,$ Theorem \ref{thm uTg} implies $u\phi_m$ is a maximal vector for the minimal Toeplitz kernel $\Ker T_{F\overline{u}/u_o}$ containing   $u\Ker T_F$. That is, $$u\Ker T_F\subseteq {\rm K_{\min}} (u\phi_m)= \Ker T_{F\overline{u}/u_o}.$$
So Theorem \ref{minimal kernel} gives $$C_\psi (u\Ker T_F)\subseteq C_\psi( \Ker T_{F\overline{u}/u_o})\subseteq \Ker T_H$$ with $H$ defined in \eqref{hpsi}. This can also be proved by the  facts  $ C_\psi(u\Ker T_F)=(u\circ \psi) C_\psi(\Ker T_F)$ and then Theorem \ref{minimal kernel} together with Theorem \ref{thm uTg} imply $$(u\circ \psi) C_\psi(\Ker T_F)\subseteq (u\circ \psi) \Ker T_{(F\circ \psi)\psi/z}\subseteq  \Ker T_H.$$   Next we show $\Ker T_H$ is the minimal Toeplitz kernel containing $C_\psi(u\Ker T_F).$

By decomposing $\phi_m=IO$ with $I$ inner and $O$ outer in $H^2,$ we obtain that \begin{eqnarray}\Ker T_F=\Ker T_{\overline{z}\overline{\phi_m}/O}\label{Tg}.\end{eqnarray}
Next we consider the minimal kernel ${\rm K_{\min}}(C_{\psi}(u\phi_m))$ for the maximal vector $ C_\psi(u\phi_m)\in C_\psi(u \Ker T_F)\subseteq H^2.$ Theorem \ref{mini kernel} implies \begin{eqnarray}  {\rm K_{\min}}(C_{\psi}(u\phi_m))=\Ker T_{\frac{\overline{z}(\overline{ u}\circ \psi)(\overline{\phi_m}\circ \psi)} {(u_o\circ\psi) (O\circ \psi)}}. \label{Tkmin}\end{eqnarray}

Since $C_\psi(u\phi_m)\in C_\psi(u \Ker T_F),$ any Toeplitz kernel containing $C_\psi(u\Ker T_F)$ must also contain
${\rm K_{ min}}(C_{\psi}(u\phi_m))$. Thus we conclude ${\rm K_{ min}}(C_{\psi}(u\phi_m))$ is the minimal kernel containing
$C_\psi(u \Ker T_F)$. And then we show ${\rm K_{ min}}(C_{\psi}(u\phi_m))=\Ker T_H$ with $H$ given in \eqref{hpsi}.

Using \eqref{Tg}, Theorem \ref{equalkernel} yields that
$$F=\overline{z}\frac{\overline{\phi_m}}{O}\frac{\overline{p}}{\overline{q}}$$ with $p, q\in  H^2$ outer. This further implies
$$\psi (F\circ\psi)= \frac{\overline{\phi_m}\circ \psi} {O\circ\psi}\frac{\overline{p}\circ \psi} {\overline{q} \circ\psi}$$ with $p, q\in  H^2$ outer. Therefore,
\begin{eqnarray}\Big(\frac{\overline{z}(\overline{ u}\circ \psi)(\overline{\phi_m}\circ \psi )}{(u_o\circ\psi) (O\circ \psi)}\Big)\frac{\overline{p} \circ \psi}{\overline{q} \circ\psi} = \frac{\overline{z}(\overline{ u}\circ \psi) }{(u_o\circ\psi)}\Big(\frac{\overline{\phi_m}\circ \psi} { O\circ \psi}
\frac{\overline{p}\circ \psi}{\overline{q}\circ\psi} \Big)= \frac{(\overline{u}\circ \psi)} {z (u_o\circ\psi) }\psi (F\circ\psi)=H.\label{equal1}\end{eqnarray}
Since   $p \circ\psi$ and $ q\circ\psi$ are outer functions, and using  \eqref{Tkmin}, \eqref{equal1}  and Theorem \ref{equalkernel}, we conclude that $${\rm K_{ min}}(C_{\psi}(u\phi_m))=\Ker T_H,$$ which is a minimal Toeplitz kernel containing $C_\psi (u \Ker T_F).$ This completes the proof.
\end{proof}

Let $\psi$ be an automorphism; then  $C_\psi(u\Ker T_F)=\Ker T_H$ can be characterized as below.

 \begin{prop}\label{prop equal3} Let $F\in L^\infty\setminus\{0\}$ and $u\in H^2$ be outer such that $\Ker T_F\neq \{0\}$ and $u\Ker T_F\subseteq H^2.$ For an automorphism $\psi$ on $\mathbb{D}$, $C_\psi(u\Ker T_F)=\Ker T_{H}$  if and only if $u\circ \psi\in \mathcal{C}(\Ker T_{(F\circ \psi)\psi/z})$ and $u^{-1}\circ \psi\in \mathcal{C}(\Ker T_H)$  with \begin{eqnarray*}H=\psi (F\circ\psi) \frac{(\overline{u}\circ \psi)} {z (u\circ\psi)}. \end{eqnarray*}\end{prop}

 \begin{proof} Noting that  $C_\psi(u\Ker T_F)=(u\circ \psi) C_\psi(\Ker T_F)$,  Theorem \ref{minimal kernel} implies $$C_\psi(\Ker T_F)=\Ker T_{(F\circ \psi)\psi/z}$$ for an automorphism $\psi$.  Then   $C_\psi(u\Ker T_F)=\Ker T_H$ is equivalent to  $$(u\circ \psi) \Ker T_{(F\circ \psi)\psi/z}=\Ker T_H.$$  So the desired equivalence follows from Theorems \ref{wequalkernel} and \ref{thm uTg1}.\end{proof}

A sufficient condition on $u$ is similarly established in the proposition below.

\begin{prop} \label{prop uTg2}  Let $F\in L^\infty\setminus\{0\}$ and $u\in H^2$ be outer such that $\Ker T_F\neq \{0\}$ and $u\Ker T_F\subseteq H^2.$ For an automorphism $\psi$ on $\mathbb{D}$,  if $u^{-1}\in H^\infty,$  it follows that
\begin{eqnarray*}C_\psi(u \Ker T_F) =\Ker T_{\frac{\psi (F\circ\psi)} {z(u\circ\psi) }}.\label{equal2}\end{eqnarray*}\end{prop}

\begin{proof} From the assumptions, Theorem \ref{minimal kernel} implies that $$C_\psi(\Ker T_F)=\Ker T_{(F\circ\psi)\psi/z}.$$  Further considering $u$ is outer and $u^{-1}\in H^\infty,$ it yields that $u\circ \psi$ is outer and $u^{-1}\circ \psi\in H^\infty.$ These facts together with Proposition \ref{prop equal} imply that  $$(u\circ \psi) \Ker T_{(F\circ\psi)\psi/z}=\Ker T_{\frac{(F\circ\psi)\psi}{z(u\circ \psi)}},$$ ending the proof.\end{proof}

In particular, every subspace of the form $C_\psi(uK_\theta)$ where $\theta$ inner and $u\in H^2$,
is contained in a minimal Toeplitz kernel -- this is a nontrivial generalization of \cite[Corollary 5.2]{CaP5}.
 \begin{cor}\label{cor umodel} Every subspace of $H^2$ of the form $C_{\psi}(uK_\theta)$, where $u\in H^2$, $\theta$ and $\psi$ are  inner functions,  is contained in  a minimal Toeplitz kernel $\Ker T_{ \psi (\overline{\theta}\circ\psi) \frac{(\overline{u}\circ \psi)} {z (u_o\circ\psi)}}$, where $u_o$ is the outer factor of $u.$\end{cor}

As an extension of Theorem \ref{containing-model}, we   formulate the minimal model space containing $C_\psi(uK_\theta)$ with inner functions $\theta$ and  $u.$

\begin{thm}\label{thm:3.8} Let $u$, $\theta$ and $\psi$ be three inner functions such that $K_\theta\neq \{0\}$, and define \begin{eqnarray*}\eta(z)=\left\{
                                    \begin{array}{ll}
                                      z(\theta  \circ \psi)(u\circ\psi)/ \psi, & \theta(0)=0\;\mbox{or}\; u(0)=0, \\
(\theta  \circ \psi)(u\circ\psi), & \theta(0)\neq 0, u(0)\neq 0\;\mbox{and}\; \psi(0)=0,\\
             z ( \theta  \circ \psi)(u\circ\psi), & \theta(0)\neq 0,\;\;u(0)\neq 0\;\mbox{and}\; \psi(0)\neq 0.
                                    \end{array}
                                  \right.\end{eqnarray*} Then $\eta$ is inner and the composition operator $$C_\psi:\;uK_\theta\rightarrow K_\eta$$ is well-defined and bounded. Moreover, $K_\eta$ is the smallest  model space containing  $C_\psi(uK_\theta)$.\end{thm}

\begin{proof} Since $K_\eta$ is the minimal model space  containing $C_\psi(uK_\theta)$, Corollary \ref{cor umodel} implies  that
 $\Ker T_{\psi(\overline{\theta}\circ\psi) (\overline{u}\circ \psi)/z}  \subseteq K_\eta=\Ker T_{\overline{\eta}}.$  Using Proposition \ref{prop FG}, it follows that  $$\Ker T_{\psi(\overline{\theta}\circ\psi) (\overline{u}\circ \psi)/z} \subseteq
\Ker T_{\overline{\eta}}\;\;\mbox{if and only if}\; \;\frac{\overline{\eta}z}{\psi(\overline{\theta} \circ \psi)
(\overline{u}\circ \psi)}\in \overline{\mathcal{N}_+}.$$
This means that $$\frac{\eta \psi }{z ( \theta  \circ \psi)( u\circ \psi)}\in \mathcal{N}_+,$$  which is unimodular.  Thus the definition of the Smirnov class implies $$\frac{\eta\psi }{z ( \theta  \circ \psi)( u\circ \psi)}\in H^\infty\subseteq H^2.$$ Hence we must find the minimal inner function $\phi$ such that
$$\eta(z)=\frac{z ( \theta  \circ \psi)(u\circ\psi)}{ \psi } \phi \in H^2.$$

For the case $\theta(0)=0$ or $u(0)=0,$ then $\psi$ divides $\theta\circ\psi$ or $u\circ\psi$,  and so take $\phi(z)=1$, which verifies $\eta(z)=z ( \theta  \circ \psi)(u\circ\psi)/ \psi.$

For the case $\theta(0)\neq 0$ and $u(0)\neq 0,$ we divide into two subcases: if $\psi(0)=0,$ we take $\phi=\psi/z$, it follows $\eta(z)= ( \theta  \circ \psi)(u\circ\psi);$  if   $\psi(0)\neq 0,$ we take $\phi=\psi$ it follows $\eta(z)=z (\theta  \circ \psi)(u\circ\psi),$ ending the proof.  \end{proof}
\begin{rem}Letting $u=1$ in Theorem \ref{thm:3.8}, Theorem \ref{containing-model} follows. \end{rem}
\subsection{The minimal Toeplitz kernel containing $uC_\psi(\Ker T_F)$}

To begin, we present a lemma about multipliers between minimal Toeplitz kernels. We then use it to formulate the minimal Toeplitz kernel containing the image of $\Ker T_F$ under the weighted composition operator $uC_\psi.$

 \begin{lem}\label{lem Kmin} If $u$, $v$, $uv\in H^2$
then $u K_{\min}(v)\subseteq K_{\min}(uv)$. \end{lem}

\begin{proof} Write $u = u_iu_o$ and $v= v_iv_o$ in their inner--outer decompositions. By Theorem \ref{mini kernel}, given any $h\in K_{\min}(v)$, it holds that $h \in \Ker T_{\overline{v}_i\overline{v}_o \overline{z}/v_o}$. And then
 $$\frac{\overline{u}_i\overline{u}_o\overline{v}_i \overline{v}_o \overline{z}}{u_ov_o}uh=\frac{ \overline{u}_o\overline{v}_i \overline{v}_o \overline{z}}{ v_o} h= \overline{u}_o\frac{\overline{v}_i \overline{v}_o \overline{z}}{ v_o} h\in\overline{zH^2}.$$
 That is $uh \in \Ker T_{\overline{u}_i\overline{u}_o \overline{v}_i \overline{v}_o\overline{z}/(u_ov_o)}=K_{\min}(uv),$ ending the proof. \end{proof}

 Now we can find the minimal  Toeplitz kernel containing $uC_\psi( \Ker T_F)$.

\begin{thm} \label{thm uTg2}  Let $F\in L^\infty\setminus\{0\}$,  $\psi$ be inner function and $u\in H^2$ such that $\Ker T_F\neq \{0\}$ and $uC_\psi( \Ker T_F)\subseteq  H^2$. Then there exists a minimal Toeplitz kernel containing $uC_\psi( \Ker T_F)$, which is $\Ker T_{G}$ with \begin{eqnarray}G=\psi (F\circ\psi) \frac{ \overline{u} } {z u_o} \label{MG}\end{eqnarray}and $u_o$ is the outer factor of $u.$ Moreover, if $\phi_m$ is a maximal vector for $\Ker T_F,$ then $u(\phi_m\circ \psi)$ is a maximal vector for $\Ker T_G.$ \end{thm}

\begin{proof} Let $\phi_m$ be a maximal vector for $\Ker T_F$. Theorem \ref{minimal kernel} implies $K_{\min}(\phi_m \circ \psi)$ is  the minimal Toeplitz kernel containing  $C_\psi(\Ker T_F)$.   Lemma \ref{lem Kmin} further implies that $$uC_\psi( \Ker T_F)\subseteq uK_{\min}(\phi_m\circ \psi)\subseteq K_{\min}(u(\phi_m\circ\psi)).$$
Since $u(\phi_m \circ\psi)\in uC_\psi(\Ker T_F),$ any Toeplitz kernel containing $uC_\psi(\Ker T_F)$ must also contain
${\rm K_{ min}}(u(\phi_m\circ\psi))$. Thus we conclude ${\rm K_{ min}}(u(\phi_m\circ\psi))$ is the minimal kernel containing
$uC_\psi( \Ker T_F)$. Next we  prove that ${\rm K_{ min}}(u(\phi_m\circ\psi))=\Ker T_G$ with $G$ given in \eqref{MG}.

 Decomposing $\phi_m=IO$ with $I$ inner and $O$ outer in $H^2,$ it follows that \begin{eqnarray}\Ker T_F=\Ker T_{\overline{z}\overline{\phi_m}/O}\label{Tgg}.\end{eqnarray} Meanwhile, Theorem \ref{mini kernel} implies \begin{eqnarray}  {\rm K_{\min}}( u(\phi_m\circ \psi))=\Ker T_{\frac{\overline{z}\overline{ u}(\overline{\phi_m}\circ \psi)} {u_o (O\circ \psi)}}. \label{Tkminn}\end{eqnarray}

Using \eqref{Tgg}, Theorem \ref{equalkernel} yields that
$$F=\overline{z}\frac{\overline{\phi_m}}{O}\frac{\overline{p}}{\overline{q}}$$ with $p, q\in  H^2$ outer. This further indicates
$$F\circ\psi= \overline{\psi}\frac{\overline{\phi_m}\circ \psi} {O\circ\psi}\frac{\overline{p}\circ \psi} {\overline{q} \circ\psi}$$ with $p, q\in  H^2$ outer. Therefore, it
further gives
$$\psi (F\circ\psi)= \frac{\overline{\phi_m}\circ \psi} {O\circ\psi}\frac{\overline{p}\circ \psi} {\overline{q} \circ\psi}$$ with $p, q\in  H^2$ outer. So we deduce that

\begin{eqnarray}\Big(\frac{\overline{z}\overline{ u}(\overline{\phi_m}\circ \psi )}{u_o (O\circ \psi)}\Big)\frac{\overline{p} \circ \psi}{\overline{q} \circ\psi} = \Big(\frac{\overline{\phi_m}\circ \psi} { O\circ \psi}
\frac{\overline{p}\circ \psi}{\overline{q}\circ\psi} \Big)\frac{\overline{z}\overline{ u}}{u_o}= \psi (F\circ\psi)\frac{\overline{u} } {z u_o  }=G.\label{equal12}\end{eqnarray}

Since  $p \circ\psi$ and $ q\circ\psi$ are outer functions,   \eqref{Tkminn}, \eqref{equal12}  and Theorem \ref{equalkernel} imply that
$${\rm K_{ min}}(u(\phi_m\circ\psi))=\Ker T_G,$$
 which is a minimal Toeplitz kernel containing $uC_\psi (\Ker T_F),$ finishing the proof.
\end{proof}

\begin{prop}\label{prop uT1} Let $F\in L^\infty\setminus\{0\}$,  $\psi\in H^2$ be an automorphism on $\mathbb{D}$  and $u$ be an outer function such that $\Ker T_F\neq \{0\}$ and $uC_\psi( \Ker T_F)\subseteq  H^2$. Then $uC_\psi( \Ker T_F)=\Ker T_{G}$ if and only if $u\in \mathcal{C}(\Ker T_{(F\circ \psi)\psi/z})$ and $u^{-1}\in \mathcal{C}(\Ker T_G)$, where $G$ is given in \eqref{MG}. \end{prop}

\begin{proof}Theorem \ref{minimal kernel} shows $C_\psi (\Ker T_F)= \Ker T_{(F\circ \psi)\psi/z}$ for the automorphism $\psi$  on $\mathbb{D}$. So $uC_\psi (\Ker T_F)=\Ker T_G$ is equivalent to $u \Ker T_{(F\circ \psi)\psi/z}=\Ker T_G$. Theorem \ref{wequalkernel} further gives the desired results. \end{proof}

\begin{rem}Letting $u^{-1}\in H^\infty$ in Proposition \ref{prop uT1}, the desired equality still holds. \end{rem}

Additionally,  the minimal Toeplitz kernel  containing the subspace  $uC_\psi(K_\theta)$ with  inner $\theta$ is given in the corollary below.
 \begin{cor}\label{cor upsimodel} Let $\theta$ be  inner and $u\in H^2$ such that $uC_\psi(K_\theta)\subseteq H^2$. For any inner function $\psi,$  there exists a minimal Toeplitz kernel $\Ker T_{ \psi (\overline{\theta}\circ\psi)  \overline{u}/(z u_o)}$ containing $uC_\psi(K_\theta)$, where $u_o$ is the outer factor of $u.$\end{cor}

This gives rise to an expression for the minimal model space containing $uC_\psi(K_\theta)$ with inner functions $u$ and $\theta$.

\begin{thm}\label{thm uCpsi} Let $u$, $\theta$ and $\psi$ be three inner functions such that $K_\theta\neq \{0\}$, and define \begin{eqnarray*}\eta(z)=\left\{
                                    \begin{array}{ll}
                                      z(\theta  \circ \psi)u/ \psi, & \theta(0)=0, \\
(\theta  \circ \psi)u, & \theta(0)\neq 0\;\mbox{and}\; \psi(0)=0,\\
             z ( \theta  \circ \psi)u, & \theta(0)\neq 0\;\mbox{and}\; \psi(0)\neq 0.
                                    \end{array}
                                  \right.\end{eqnarray*} Then $\eta$ is inner and the weighted composition operator $$uC_\psi:\;K_\theta\rightarrow K_\eta$$ is well-defined and bounded. Moreover, $K_\eta$ is the smallest  model space containing  $uC_\psi(K_\theta)$.\end{thm}

\begin{proof} We find $\eta$ such that $K_\eta$ is the minimal model space  containing $uC_\psi(K_\theta)$, Corollary \ref{cor upsimodel} implies  that
 $\Ker T_{\psi(\overline{\theta}\circ\psi)  \overline{u} /z}  \subseteq K_\eta=\Ker T_{\overline{\eta}}.$    Proposition \ref{prop FG} yields that  $$\Ker T_{\psi(\overline{\theta}\circ\psi) \overline{u} /z} \subseteq
\Ker T_{\overline{\eta}}\;\;\mbox{if and only if}\; \;\frac{\overline{\eta}z}{\psi(\overline{\theta} \circ \psi)
 \overline{u} }\in \overline{\mathcal{N}_+}.$$
That is, $$\frac{\eta \psi }{z ( \theta  \circ \psi)u}\in \mathcal{N}_+,$$  which is unimodular. Hence  the definition of the Smirnov class implies $$\frac{\eta\psi }{z ( \theta  \circ \psi)u}\in H^\infty\subseteq H^2.$$ So we need to find the minimal inner function $\phi$ such that
$$\eta(z)=\frac{z ( \theta  \circ \psi)u}{ \psi } \phi \in H^2.$$

For the case $\theta(0)=0$, it follows that $\psi$ divides $\theta\circ\psi$,  and so take $\phi(z)=1$, which gives $\eta(z)=z ( \theta  \circ \psi)u/ \psi.$

For the case $\theta(0)\neq 0$, we divide into two subcases: if $\psi(0)=0,$ we take $\phi(z)=\psi(z)/z$, and it follows that $\eta(z)= ( \theta  \circ \psi)u;$  if   $\psi(0)\neq 0,$ we take $\phi(z)=\psi(z)$ and then $\eta(z)=z (\theta  \circ \psi)u,$ finishing the proof.  \end{proof}
\begin{rem} Letting $u=1$ in Theorem \ref{thm uCpsi}, Theorem \ref{containing-model} follows.\end{rem}

\section{Maximal vectors for  Toeplitz kernels with compositional inner symbols}

In this final section, we generalise the above results in the sense of equivalence (see, e.g. \cite{CaP3,CMP4}). To be specific, we formulate the maximal vectors for several Toeplitz kernels (including model spaces)  with symbols
expressed in terms of composition operators and inner functions.

\begin{defn}\cite[Definition 2.23]{CaP3}\label{functionG} If  $G_1, G_2\in L^\infty$, we say that $G_1\sim G_2$ if and only if there are functions $H_+\in \mathcal{G}H^\infty$ and  $ H_- \in \mathcal{G}\overline{H^\infty},$ such that \begin{align}G_1=H_-G_2H_+.\label{G12}\end{align}\end{defn}
\begin{defn}\cite[Definition 2.24]{CaP3}\label{kernelG} If $G_1, G_2\in L^\infty\setminus\{0\}$, such that $\Ker T_{G_1}$ and $\Ker T_{G_2}$ are nontrivial, we say that
 $\Ker T_{G_1}\sim \Ker T_{G_2}$ if and only if $$ \Ker T_{G_1}=H_+\Ker T_{G_2}\;\;\mbox{with}\;H_+\in \mathcal{G} H^\infty.$$  \end{defn}

Definitions \ref{functionG} and \ref{kernelG} imply that if $G_1\sim G_2$ via \eqref{G12} then
$\Ker T_{G_1}=H_+^{-1}\Ker T_{G_2}$, that is, $\Ker T_{G_1}\sim \Ker T_{G_2}.$ In the sequel, we want to express the maximal vector for a Toeplitz kernel $\Ker T_{h_-(g\circ \psi)h_+}=h_+^{-1}\Ker T_{g\circ \psi}$ in terms of a maximal vector of $\Ker T_g,$ where  $\psi$  is inner, $h_+\in \mathcal{G}H^\infty\;\mbox{and}\;\;h_-\in \mathcal{G}\overline{H^\infty}.$ The first two propositions concern the maximal vectors for Toeplitz kernels   $\Ker T_{\overline{z}h_- (g\circ \psi) h_+}=h_+^{-1}\Ker T_{\overline{z} (g\circ \psi)}$ and   $\Ker T_{\overline{z}\psi h_- (g\circ \psi) h_+}=h_+^{-1}\Ker T_{\overline{z}\psi (g\circ \psi)}$.

\begin{prop} \label{prop h+-}If $k$ is a maximal vector for $\Ker T_g,$ where $g\in L^\infty,$ then  $h_+^{-1} (k\circ\psi)\psi$ is a maximal vector for $\Ker T_{\overline{z}h_- (g\circ \psi) h_+}=h_+^{-1}\Ker T_{\overline{z} (g\circ \psi)}$ for every inner function $\psi$ and every $h_+\in \mathcal{G}H^\infty,\;h_-\in \mathcal{G}\overline{H^\infty}.$ That is,   \begin{align} {\rm K_{\min}}(h_+^{-1} (k\circ \psi)\psi)=\Ker T_{\overline{z} h_-(g\circ \psi) h_+}=h_+^{-1}\Ker T_{\overline{z} (g\circ \psi)}.\label{Kmin1} \end{align} \end{prop}

\begin{proof} From Theorem \ref{maximal1}, if ${\rm K_{\min}}(k)=\Ker T_g,$ then $gk=\overline{zp}$ for an outer function $p\in H^2$. Therefore we further have $$\psi (g\circ \psi)(k\circ \psi)= \overline{p}\circ \psi.$$It holds that $h_+^{-1}( k\circ \psi)\psi \in H^2$, which satisfies
 \begin{eqnarray}(\overline{z} h_-(g\circ \psi)h_+)(h_+^{-1} (k\circ \psi)\psi )=\overline{z}h_-(\overline{p}\circ \psi). \label{gpsi1}\end{eqnarray}
Since $p\circ \psi$ is outer and $h_-\in \mathcal{G}\overline{H^\infty}$, the above equation, together with Theorem \ref{maximal1}, implies  the required result.
\end{proof}

And then \eqref{Kmin1} implies $ {\rm K_{\min}}(  (k\circ \psi)\psi)=\Ker T_{\overline{z} (g\circ \psi)}$ due to $h_+\in \ma{G}H^\infty.$ This reduces to a corollary on model spaces.

\begin{cor}\label{cor Model1}Let $\theta$ and $\psi$ be two inner functions such that $K_\theta\neq \{0\}$; then $C_\psi(S^*\theta)\psi$ is a maximal vector for $K_{z(\theta\circ \psi)}.$ That is,  $$K_{\min}(C_\psi(S^*\theta)\psi)=K_{z(\theta\circ \psi)}.$$ \end{cor}

 Replacing $h_+^{-1}( k\circ \psi)\psi \in H^2$ by $h_+^{-1}( k\circ \psi)\in H^2$ in \eqref{gpsi1}, we deduce the following proposition.
\begin{prop} \label{prop h+-1}If $k$ is a maximal vector for $\Ker T_g,$ where $g\in L^\infty,$ then  $h_+^{-1} (k\circ\psi) $ is a maximal vector for $\Ker T_{\overline{z}\psi h_- (g\circ \psi) h_+}=h_+^{-1}\Ker T_{\overline{z}\psi  (g\circ \psi) }$ for every inner function $\psi$ and every $h_+\in \mathcal{G}H^\infty,\;h_-\in \mathcal{G}\overline{H^\infty}.$ That is,   $$ {\rm K_{\min}}(h_+^{-1} (k\circ \psi) )=\Ker T_{\overline{z}\psi h_-(g\circ \psi) h_+}=h_+^{-1}\Ker T_{\overline{z}\psi  (g\circ \psi) }.$$\end{prop}

And then it holds that $ {\rm K_{\min}}(k\circ \psi)= \Ker T_{\overline{z}\psi  (g\circ \psi) }$ due to $h_+\in \ma{G}H^\infty.$ This indicates Corollary \ref{cor models} by letting $g:=\overline{\theta}$ and $k:=S^*\theta$ with an inner $\theta$.

For the case $\psi(0)=0,$  we replace $h_+^{-1}( k\circ \psi)\psi \in H^2$ by $h_+^{-1}( k\circ \psi)\psi/z\in H^2$ in \eqref{gpsi1}, we deduce the proposition below.
\begin{prop} \label{hg}If $k$ is a maximal vector for $\Ker T_g,$ where $g\in L^\infty,$ then  $h_+^{-1} (k\circ\psi)\psi/z$ is a maximal vector for $\Ker T_{h_- (g\circ \psi) h_+}=h_+^{-1}\Ker T_{g\circ \psi}$ for every inner function $\psi$ with $\psi(0)=0$ and every $h_+\in \mathcal{G}H^\infty,\;h_-\in \mathcal{G}\overline{H^\infty}.$ That is,
 $$ {\rm K_{\min}}(h_+^{-1} (k\circ \psi)\psi/z)=\Ker T_{h_-(g\circ \psi) h_+}=h_+^{-1}\Ker T_{g\circ \psi}.$$ \end{prop}

The above formula ensures $ {\rm K_{\min}}( (k\circ \psi)\psi/z)= \Ker T_{g\circ \psi}$ when $\psi(0)=0$. Taking  $g:=\overline{\theta}$ and $k:=S^* \theta$ with an inner $\theta$,  a corollary follows on  model spaces.

\begin{cor}\label{cor 0} Let $\eta$  and $\psi$ be  inner functions with $\psi(0)=0$; then  $C_\psi(S^* \eta)\psi/z$ is a maximal vector for $K_{\eta\circ\psi}$. That is,  $$ {\rm K_{\min}}(C_\psi(S^* \eta)\psi/z)=K_{\eta\circ\psi}.$$\end{cor}

\begin{rem}
Additionally, Corollary \ref{cor 0} can be used to give another family of  maximal vectors for the model space $K_\theta,$ which is an explicit special case of Theorem \ref{maximal1}.

 Let $\theta$   be an inner function  with $\theta(0)=0,$  $h_+=(1+\overline{a}\beta)^{-1}\in \mathcal{G}H^\infty$ and $h_-=1+a\overline{\beta}\in \mathcal{G}\overline{H^\infty}$ with $\beta=(\theta-a)/(1-\overline{a}\theta)$ and $|a|<1$; then
    \begin{eqnarray}\overline{h_-^{-1}}\frac{1}{1-\overline{a}\theta } \frac{\theta}{z}=\frac{h_+}{1-\overline{a}\theta  } \frac{\theta}{z}\label{matheta}\end{eqnarray} is a maximal vector for $K_{\theta}$
    for almost all such $a.$ \end{rem}
\begin{proof}
Recall the automorphism on $\mathbb{D}$ is $\varphi_a(z)=(a-z)/(1-\overline{a}z)$ with $|a|<1.$ Then each inner function $\theta$ can be decomposed into $$\theta=h_-\beta h_+$$ with $\beta$ inner (see, e.g. \cite{Nik}), where
$$\beta=\frac{\theta-a}{1-\overline{a}\theta}=-\varphi_a\circ \theta.$$
Then we deduce  $$\overline{\theta}=\overline{h_+}(-\overline{\varphi_a}\circ \theta)\overline{h_-}$$ with $\overline{h_+}\in \mathcal{G}\overline{H^\infty}$ and $\overline{h_-}\in \mathcal{G} H^\infty.$ Definition \ref{functionG} implies $\overline{\theta}\sim (-\overline{\varphi_a}\circ \theta)$ and $K_{\theta}=\overline{h_-}^{-1}K_{-\overline{\varphi_a}\circ \theta}.$ So letting $\eta:=-\varphi_a$ and $\psi:=\theta$ in  Corollary \ref{cor 0} and using $S^*\eta=S^*(-\varphi_a)=(1-|a|^2)/(1-\bar{a}z)$, it yields that $$\overline{h_-}^{-1} ((S^*\eta)\circ \psi)\psi/z=\overline{h_-}^{-1}\frac{1-|a|^2}{1-\bar{a}\theta}
\frac{\theta}{z}.$$ Then the function in \eqref{matheta} is the desired maximal vector for $K_\theta$ with $\theta(0)=0$.
\end{proof}

Using the $S^*$-invariance of model spaces, we obtain an extension of Corollary \ref{cor Model1}.  An inner function $\theta$ is called the least common multiple (LCM) of the family  of inner functions $\{\theta_i:\;i\in I\}$, if each $\theta_i$ divides $\theta$, and $\theta$ divides every other inner function that is divisible by each $\theta_i.$

\begin{thm}\label{model case}  Let $\psi$  and $\theta_i$ be inner functions for $i=1,2,\cdots,n$. Then there exists a minimal kernel $K$ containing $\{C_\psi(S^*\theta_i)\psi:\;i=1,2,\cdots,n\}$ and for $\theta={\rm LCM} (\theta_1, \; \theta_2,\;\cdots,\;\theta_n),$ we have
$$K=K_{z(\theta\circ \psi)}={\rm clos}_{H^2}(K_{z(\theta_1\circ \psi)}+\cdots+K_{z(\theta_n\circ \psi)})=K_{z(\theta_i\circ \psi)}\oplus z(\theta_i\circ \psi)K_{(\theta\circ \psi)(\overline{\theta_i}\circ \psi)}.$$ \end{thm}

\begin{proof} Corollary \ref{cor Model1} implies $C_\psi(S^*\theta_i)\psi$ is a maximal vector for $K_{z(\theta_i\circ\psi)}$ and  \cite[Theorem 4.3]{CaP2} indicates $K_{\widetilde{\theta}}$ is the minimal model space  containing $\{C_\psi(S^*\theta_i)\psi:\;i=1,2,\cdots,n\}$ with $$\widetilde{\theta}={\rm LCM}(z(\theta_1\circ \psi),\;z(\theta_2\circ \psi), \;\cdots,\; z(\theta_n\circ \psi)). $$
Since $\widetilde{\theta}=z(\theta\circ\psi),$ the desired results follow.
  \end{proof}

Motivated by Theorem \ref{model case}, it is natural to inquire whether the sum of two Toeplitz kernels (or even nearly $S^*$-invariant subspaces) is nearly $S^*$-invariant. First, the well-known lemma by Coburn provides a very interesting property possessed by the class of Toeplitz operators.

\begin{lem}\cite[Lemma B4.5.6]{Nik} Let $\varphi\in L^\infty,$ then at least one of the kernels $\Ker T_\varphi$ or $\Ker T_{\varphi}^*=\Ker T_{\overline{\varphi}}$ has to be trivial.  \end{lem}

\begin{cor}Let $\varphi\in L^\infty,$ then $\Ker T_\varphi+\Ker T_{\varphi}^*$ is always a Toeplitz kernel, so it is nearly $S^*$-invariant. \end{cor}

A similar result holds  for arbitrary nearly $S^*$-invariant subspaces $M_1$ and $M_2$ in $H^2$ when $ M_1\cap M_2\neq \{0\}.$

\begin{prop} Let $M_1$ and $M_2$ be two nearly $S^*$-invariant subspaces with $ M_1\cap M_2\neq \{0\},$ then $M_1+M_2$ is nearly $S^*$-invariant.\end{prop}

\begin{proof} Since $M_1\cap M_2\neq \{0\},$  there is a nonzero function $g\in M_1\cap M_2.$ We may assume $g(0)\neq 0$ and then if $(f_1+f_2)(0)=0$, we have a $\lambda\in \mathbb{C}$ such that $(f_1-\lambda g)(0)=0$ and $(f_2+\lambda g)(0)=0.$ Then if we denote $h_1:=(f_1-\lambda g)/z$ and $h_2:=(f_2+\lambda g)/z$, it follows that $h_i\in M_i$ due to the near $S^*$-invariance of $M_i$  for $i=1, 2$, so $S^*(f_1+f_2)=h_1+h_2\in M_1+M_2,$ completing the proof.
\end{proof}

\begin{exm}
However,  the sum of two Toeplitz kernels need not be nearly $S^*$-invariant in $H^2$.  One Toeplitz kernel is obtained from Proposition \ref{prop equal} by taking $F(z)=\bar{z}$ and $u(z)=(z+2)^2$ with $u^{-1}\in H^\infty$; then it follows that $$(z+2)^2K_z=
\Ker T_{\bar{z} (\bar{z}+2)^2/(z+2)^2}=\Ker T_{\bar{z}^3 b(z)^2},$$
where $b$ is the Blaschke product $b(z)=\frac{z+1/2}{1+z/2}.$
Let $M_1=\mathbb{C}=K_z$ and $M_2=\Ker T_g$ with $g(z):=\overline{z}^3 b(z)^2$ be
the two Toeplitz kernels satisfying $M_1\cap M_2=\{0\}$. Then $M_1+M_2$ is not nearly $S^*$-invariant, since it contains $z^2+4z$ but not $z+4.$ \end{exm}
\section*{Declarations}

\textbf{Funding}\;\; This work was supported  by the National Natural Science Foundation of China (Grant No. 12471126).

\textbf{Data availability}\;\; Not applicable.

\textbf{Conflict of interest} The authors have no relevant financial or non-financial interests to disclose.

\end{document}